\documentclass[12pt,a4paper]{article}
\usepackage{latexsym,amssymb,amsfonts,amsmath,amsthm,nccmath,float,enumitem,setspace,bm,authblk,appendix}
\usepackage[usenames,dvipsnames]{xcolor}
\usepackage[hidelinks]{hyperref}
\usepackage[margin=2cm]{geometry}

\hypersetup{
    colorlinks,
    linkcolor={red!80!black},
    citecolor={blue!80!black},
    urlcolor={blue!80!black}
}

\newtheorem{Theorem}{Theorem}[section]

\newtheorem{Corollary}{Corollary}[section]
\newtheorem{Conjecture}[Theorem]{Conjecture}

\newtheorem{Lemma}{Lemma}[section]

\newcommand{\Z}{\mathbb{Z}}

\newcommand{\B}{\mathcal{B}}
\newcommand{\F}{\mathcal{F}}

\def \leq {\leqslant}
\def \geq {\geqslant}

\let\oldproofname=\proofname
\renewcommand{\proofname}{\rm\bf{\oldproofname}}

\numberwithin{equation}{section}

\allowdisplaybreaks[4]

\begin{document}

\title{The asymptotic existence of BIBDs having a nesting}


\author[a]{Xinyue Ming}
\author[a]{Tao Feng}
\author[a]{Menglong Zhang}
\affil[a]{School of Mathematics and Statistics, Beijing Jiaotong University, Beijing, 100044, P.R. China}
\affil[ ]{xyming@bjtu.edu.cn; tfeng@bjtu.edu.cn; mlzhang@bjtu.edu.cn}

\date{}
\maketitle

\footnotetext{Supported by NSFC under Grant 12271023}


\begin{abstract}
Let $X$ be a set of $v$ points, and $\B$ be a collection of $k$-subsets of $X$ called blocks. A pair $(X,\B)$ is called a $(v,k,\lambda)$-BIBD (resp. $(v,k,\lambda)$-packing) if every pair of distinct elements of $X$ is contained in exactly (resp. at most) $\lambda$ blocks of $\B$. A $(v,k,\lambda)$-BIBD $(X,\B)$ can be nested if there is a mapping $\phi:\B\rightarrow X$ such that $(X,\{B\cup\{\phi(B)\}\mid B\in\B\})$ is a $(v,k+1,\lambda+1)$-packing. A harmonious coloring of a graph is a proper vertex coloring such that any two edges receive different color pairs. A $(v,k,\lambda)$-BIBD has a (perfect) nesting if and only if its incidence graph has a harmonious (exact) coloring with $v$ colors. This paper shows that given any positive integers $k$ and $\lambda$, if $k\geq 2\lambda+2$, then for any sufficiently large $v$, every $(v,k,\lambda)$-BIBD can be nested into a $(v,k+1,\lambda+1)$-packing; and if $k=2\lambda+1$, then for any sufficiently large $v$ satisfying $v \equiv 1 \pmod {2k}$, there exists a $(v,k,\lambda)$-BIBD having a perfect nesting. Banff difference families (BDF), as a special kind of difference families (DF), can be used to generate nested designs. This paper shows that if $G$ is a finite abelian group with a large size whose number of $2$-order elements is no more than a given constant, and $k\geq 2\lambda+2$, then a $(G,k,\lambda)$-BDF can be obtained by taking any $(G,k,\lambda)$-DF and then replacing each of its base blocks by a suitable translation. This is a Nov\'{a}k-like theorem.
Nov\'{a}k conjectured in 1974 that for any cyclic Steiner triple system of order $v$, it is always possible to choose one block from each block orbit so that the chosen blocks are pairwise disjoint. Nov\'{a}k's conjecture was generalized to any cyclic $(v,k,\lambda)$-BIBDs by Feng, Horsley and Wang in 2021, who conjectured that given any positive integers $k$ and $\lambda$ such that $k\geq \lambda+1$, there exists an integer $v_0$ such that, for any cyclic $(v,k,\lambda)$-BIBD with $v\geq v_0$, it is always possible to choose one block from each block orbit so that the chosen blocks are pairwise disjoint. This paper confirms this conjecture for every $k\geq \lambda+2$.
\end{abstract}

\noindent {\bf Keywords}: nested balanced incomplete block design; Banff difference family; harmonious coloring; Nov\'{a}k's conjecture; cyclic design; $A$-perfect matching



\section{Introduction}

Let $X$ be a set of $v$ {\em points}, and $\B$ be a collection of $k$-subsets of $X$ called {\em blocks}. A pair $(X,\B)$ is called a {\em $(v,k,\lambda)$-packing} if every pair of distinct elements of $X$ is contained in at most $\lambda$ blocks of $\B$. A $(v,k,\lambda)$-packing is called a {\em $(v,k,\lambda)$-balanced incomplete block design} (briefly BIBD) if ``at most'' is replaced by ``exactly'' in the definition of a packing. Alternatively, designs are packings with the maximum number $\lambda\binom{v}{2}/\binom{k}{2}$ of blocks.


A $(v,k,\lambda)$-BIBD $(X,\B)$ can be {\em nested} if there is a mapping $\phi:\B\rightarrow X$ such that $(X,\{B\cup\{\phi(B)\}\mid B\in\B\})$ forms a $(v,k+1,\lambda+1)$-packing. The mapping $\phi$ is called a {\em nesting} of the $(v,k,\lambda)$-BIBD. If the $(v,k+1,\lambda+1)$-packing is a $(v,k+1,\lambda+1)$-BIBD, we have a {\em perfect nesting}. By comparing the number $\lambda\binom{v}{2}/\binom{k}{2}$ of blocks in a $(v,k,\lambda)$-BIBD and the largest possible number $\lfloor(\lambda+1)\binom{v}{2}/\binom{k+1}{2}\rfloor$ of blocks in a $(v,k+1,\lambda+1)$-packing, Buratti, Kreher and Stinson \cite[Lemma 1.1]{BKS2024} showed that if a $(v,k,\lambda)$-BIBD has a nesting then $k\geq 2\lambda+1$. Furthermore, they \cite[Theorem 1.2 and Lemma 1.3]{BKS2024} showed that a $(v,k,\lambda)$-BIBD has a perfect nesting only if $k=2\lambda+1$ and $v\equiv 1\pmod{2k}$.

A $(v,3,1)$-BIBD is a {\em Steiner triple system} and often written as an STS$(v)$. Stinson \cite{Stinson85} showed that there is an STS$(v)$ having a perfect nesting if and only if $v\equiv 1\pmod{6}$. Buratti, Kreher and Stinson \cite{BKS2024} proved that there is a nested $(v,4,1)$-BIBD if and only if $v\equiv 1,4\pmod{12}$ and $v\geq 13$. Note that nestings of $(v,4,1)$-BIBDs are not perfect nestings since $2\lambda+1=3<4=k$.

We shall prove the following theorem in Section \ref{sec:Thm1}, which shows that every $(v,k,\lambda)$-BIBD with $k\geq 2\lambda+2$ and large enough $v$ has a nesting.

\begin{Theorem}\label{thm:nested_BIBD}
Given any positive integers $k$ and $\lambda$ with $k\geq 2\lambda+2$, for any sufficiently large $v$, every $(v,k,\lambda)$-BIBD can be nested into a $(v,k+1,\lambda+1)$-packing.
\end{Theorem}

Nested $(v,k,\lambda)$-BIBDs have also been considered from another point of view by Buratti, Merola, Naki\'{c} and Rubio-Montiel \cite{BMNR23} very recently. A {\em harmonious coloring} of a graph is a mapping that assigns colors to the vertices of this graph such that each vertex has exactly one color, adjacent vertices have different colors, and any two edges receive different color pairs. Here a color pair for an edge is the set of colors on the vertices of this edge. The {\em harmonious chromatic number} $h(G)$ of a graph $G$ is the minimum number of colors assigned to the vertices in a harmonious coloring of $G$. This parameter was introduced by Miller and Pritikin \cite{MillerPritikin91}. It is only a slight variation of the definition given independently by Hopcroft and Krishnamoorthy \cite{HK83} and by Frank, Harary and Plantholt \cite{FHP82} where adjacent vertices are allowed to receive the same color. It was shown by Hopcroft and Krishnamoorthy \cite{HK83} that the problem of determining the harmonious chromatic number of a graph is NP-complete. We refer the reader to \cite{AGMMRRT,AAEEJR,DLR,Edwards97,Mitchem84} for more information on harmonious chromatic numbers. We also point out the updated bibliography maintained by Edwards \cite{Edwards98}.


The {\em incidence graph} or the {\em Levi graph} of a $(v,k,\lambda)$-BIBD $(X,\mathcal B)$ is a bipartite graph $G$, where the vertex set $V(G)=X\cup \mathcal B$ is partitioned into two parts $X$ and $\mathcal B$, and the set of edges is $E(G)=\{\{x,B\}\mid x\in X,B\in\mathcal B\text{~and~}x\in B\}$.
Let $(X,\mathcal B)$ be a $(v,k,\lambda)$-BIBD and $L(X,\mathcal B)$ be its Levi graph. Since any two distinct points of $X$ are contained in exactly $\lambda\geq 1$ blocks, any two different vertices $x,y\in X$ receive different colors in any harmonious coloring of $L(X,\mathcal B)$. Thus $h(L(X,\mathcal B))\geq v$. If $h(L(X,\mathcal B))=v$, any harmonious coloring of $L(X,\mathcal B)$ using exactly $v$ colors gives rise to a nesting, i.e. a $(v,k+1,\lambda+1)$-packing, by adding to the block $B$ the only point that has the same color of $B$. Conversely, the Levi graph of a $(v,k,\lambda)$-BIBD $(X,\mathcal B)$ having a nesting $\phi$ can be harmoniously colored with $v$ colors, by coloring each element of $X$ with a different color, and coloring each block $B\in\mathcal B$ with the color given to the point $\phi(B)$. Therefore, a $(v,k,\lambda)$-BIBD $(X,\mathcal B)$ has a nesting if and only if $h(L(X,\mathcal B))=v$. Applying Theorem \ref{thm:nested_BIBD} gives the following result.


\begin{Corollary}\label{cor:color Levi}
$h(L(X,\mathcal B))=v$ for every $(v,k,\lambda)$-BIBD $(X,\mathcal B)$ with $k\geq 2\lambda+2$ and large enough $v$.
\end{Corollary}

A $(v,k,\lambda)$-BIBD whose Levi graph has harmonious chromatic number equal to $v$ is called a {\em Banff design} in \cite{BMNR23}. Wilson \cite{Wilson75} showed that for any given positive integers $k$ and $\lambda$, there exist $(v, k, \lambda)$-BIBDs for all sufficiently large integer $v$ satisfying $\lambda(v-1)\equiv 0\pmod{k-1}$ and $\lambda v(v-1)\equiv 0\pmod{k(k-1)}$. Therefore, as a corollary of Theorem \ref{thm:nested_BIBD} and Corollary \ref{cor:color Levi}, we have the following asymptotic existence of $(v,k,\lambda)$-BIBDs having a nesting.

\begin{Corollary}\label{cor:nested_BIBD}
Given any positive integers $k$ and $\lambda$ with $k\geq 2\lambda+2$, for any sufficiently large $v$ satisfying $\lambda(v-1)\equiv 0\pmod{k-1}$ and $\lambda v(v-1)\equiv 0\pmod{k(k-1)}$, there exists a Banff $(v,k,\lambda)$-BIBD, which has a nesting.
\end{Corollary}


A harmonious coloring of a graph $G$ is called an {\em exact coloring} if each pair of colors appears in exactly one edge of $G$. A $(v,k,\lambda)$-BIBD has a perfect nesting if and only if its Levi graph has an exact coloring with $v$ colors (see \cite[Lemma 1.5]{BKS2024}). The following theorem, which is a natural generalization of \cite[Example 2.1]{lw}, provides an asymptotic existence result for $(v,k,\lambda)$-BIBDs having a perfect nesting.

\begin{Theorem}\label{thm:perfect_nested}
Let $\lambda$ be a positive integer and $k=2\lambda+1$. For any sufficiently large $v$ satisfying $v \equiv 1 \pmod {2k}$, there exists a $(v,k,\lambda)$-BIBD having a perfect nesting.
\end{Theorem}


Theorem \ref{thm:perfect_nested} is a special case of the following Theorem \ref{thm:(k_1__k_2)_nested}. Federer \cite{Federer72} introduced a more general concept on nested designs. Let $k_1<k_2$. A {\em $(k_2,\lambda_2;k_1, \lambda_1)$-nesting} is a $(v,k_2,\lambda_2)$-BIBD where each block has a distinguished subblock of cardinality $k_1$ and these subblocks form a $(v,k_1,\lambda_1)$-BIBD (cf. \cite[Section 4]{Jimbo93}). For such a nesting to exist, the number $\lambda_1 v(v-1)/(k_1(k_1-1))$ of blocks in the $(v,k_1,\lambda_1)$-BIBD equals the number $\lambda_2v(v-1)/(k_2(k_2-1))$ of blocks in the $(v,k_2,\lambda_2)$-BIBD, which yields $\lambda_1k_2(k_2-1)=\lambda_2k_1(k_1-1)$. We shall prove the following theorem in Section \ref{sec:Thm2}.


\begin{Theorem}\label{thm:(k_1__k_2)_nested}
Given any positive integers $k_1,k_2,\lambda_1$ and $\lambda_2$ with $\lambda_1k_2(k_2-1)=\lambda_2k_1(k_1-1)$, for any sufficiently large $v$ satisfying $\lambda_1(v-1)\equiv 0\pmod{k_1-1}$, $\lambda_1 v(v-1)\equiv 0\pmod{k_1(k_1-1)}$ and $\lambda_2(v-1)\equiv 0\pmod{k_2-1}$, there exists a $(k_2,\lambda_2;k_1, \lambda_1)$-nesting.
\end{Theorem}

Recall that a $(v,k,\lambda)$-BIBD has a nesting if and only if it is a Banff design. In order to use difference methods to construct Banff designs, Banff difference families were introduced in \cite{BMNR23}. Let $(G,+)$ be a finite (not necessarily abelian) group with the identity $0$ and let $k$ be an integer satisfying $2\leq k\leq |G|$. A $(G,k,\lambda)$-{\em difference family} (briefly DF) is a collection $\F$ of $k$-subsets of $G$ (called {\em base blocks}) such that the list $\Delta \F=\bigcup_{F\in \F}\{f-f'\mid f,f'\in F, f\neq f'\}$ covers each element of $G\setminus \{0\}$ exactly $\lambda$ times. A $(G,k,\lambda)$-DF is called a {\em Banff difference family} (briefly BDF), if all of its base blocks and their negatives are mutually disjoint. Here the {\em negative} of a base block $B$ is $-B=\{-b\mid b\in B\}$. According to the definition, each base block of a BDF does not contain $0$. For example, $\{7,8,11\}$ and $\{4,10,12\}$ form a $(\Z_{13},3,1)$-BDF. In this paper we always denote by $\Z_v$ the additive group of integers modulo $v$.

It is shown in \cite[Proposition 3.7]{BMNR23} that every Banff $(G,k,\lambda)$-DF, say $\mathcal F$, generates a Banff $(|G|,k,\lambda)$-BIBD with the block (multi)-set $\{F+g: F \in\mathcal F, g\in G\}$, which can be nested into a $(|G|,k+1,\lambda+1)$-packing with the block (multi)-set $\{(F+g)\cup\{g\}: F \in\mathcal F, g\in G\}$. For convenience, we call $F+g=\{f+g\mid f\in F\}$ a {\em translation} of $F$. We are to prove the following theorem in Section \ref{sec:banff DF}.

\begin{Theorem}\label{thm:DF-Banff_DF}
For any positive integers $C$, $k$ and $\lambda$ with $k\geq 2\lambda+2$, there exists an integer $v_0$ such that the following holds. Let $G$ be a finite abelian group such that $|G|>v_0$ and the number of $2$-order elements in $G$ is at most $C$. For any $(G,k,\lambda)$-DF, it is always possible to choose one base block from the set of translations of each base block so that the chosen base blocks form a $(G,k, \lambda)$-BDF.
\end{Theorem}

Theorem \ref{thm:DF-Banff_DF} reduces the existence of $(G, k, \lambda)$-BDFs to the existence of $(G, k, \lambda)$-DFs when $G$ is a finite abelian group with a large size and $k\geq 2\lambda+2$.
By \cite[Theorem 5]{Wilson72}, there exists an $((\mathbb{F}_q,+),k,\lambda)$-DF over the additive group of the finite field $\mathbb{F}_q$ of order $q$ for any sufficient large prime power $q$ and $\lambda(q-1)\equiv 0 \pmod {k(k-1)}$.
Therefore, as a corollary of Theorem \ref{thm:DF-Banff_DF}, we have the following asymptotic existence result on BDFs, which is a generalization of \cite[Theorem 5.4]{BMNR23}.

\begin{Corollary}\label{cor:some_BDFs}
Let $k$ and $\lambda$ be positive integers such that $k\geq 2\lambda+2$. For any sufficiently large odd prime power $q$ satisfying $\lambda(q-1)\equiv 0 \pmod {k(k-1)}$, an $((\mathbb{F}_q,+),k,\lambda)$-BDF exists.
\end{Corollary}

Theorem \ref{thm:DF-Banff_DF} is a Nov\'{a}k-like theorem. An {\em automorphism} of a $(v, k, \lambda)$-BIBD $(X,\mathcal  B)$ is a permutation on $X$ leaving $\mathcal B$ invariant.
A $(v,k,\lambda)$-BIBD $(X,\mathcal  B)$ is said to be {\em cyclic} if it admits an automorphism consisting of a cycle of length $v$. Without loss of generality we identify $X$ with $\mathbb{Z}_v$. The blocks of a cyclic $(v,k,\lambda)$-BIBD can be partitioned into {\em orbits} under $\mathbb{Z}_v$. Choose any fixed block from each orbit and then call them {\em base blocks}. Nov\'{a}k \cite{Novak75} conjectured that for any cyclic STS$(v)$ with $v\equiv 1\pmod{6}$, it is always possible to find a set of $(v-1)/6$ mutually disjoint base blocks which come from different block orbits to form a $(\mathbb{Z}_v, 3, 1)$-DF (see also \cite[Remark 16.22]{AB06} or \cite[Work point 22.5.2]{CR99}). Feng, Horsley and Wang \cite{FHW2021} gave an asymptotic solution to this conjecture as a corollary of the following more general result, and they made a more general conjecture.

\begin{Theorem}\label{thm:novak_conj_solution}{\rm \cite[Theorem 3]{FHW2021}}
Let $k$ and $\lambda$ be fixed positive integers such that $k\geq 2\lambda + 1$. There exists an integer $v_0$ such that, for any cyclic $(v,k,\lambda)$-BIBD with $v\geq v_0$, it is always possible to choose one block from each block orbit so that the chosen blocks are pairwise disjoint.
\end{Theorem}

\begin{Conjecture}\label{conj:novak_conj_gener}{\rm \cite[Conjecture 3]{FHW2021}}
Let $k$ and $\lambda$ be fixed positive integers such that $k\geq \lambda + 1$. There exists an integer $v_0$ such that, for any cyclic $(v,k,\lambda)$-BIBD with $v\geq v_0$, it is always possible to choose one block from each block orbit so that the chosen blocks are pairwise disjoint.
\end{Conjecture}

In Section \ref{sec:Nov} we shall prove the following theorem that confirms the above conjecture for any $k\geq \lambda+2$ by using a similar technique to those for the proofs of Theorems \ref{thm:nested_BIBD} and  \ref{thm:DF-Banff_DF}.

\begin{Theorem}\label{thm:novak_conj_solution-main}
Let $k$ and $\lambda$ be fixed positive integers such that $k\geq \lambda + 2$. There exists an integer $v_0$ such that, for any cyclic $(v,k,\lambda)$-BIBD with $v\geq v_0$, it is always possible to choose one block from each block orbit so that the chosen blocks are pairwise disjoint.
\end{Theorem}

\section{Nested BIBDs}

We open this section with a result about the existence of an $A$-perfect matching in a bipartite hypergraph.

A {\em $($multi$)$-hypergraph} consists of a pair $(V,E)$ where $V$ is a set whose elements are called {\em vertices} and $E$ is a (multi)-set of subsets of $V$ called {\em edges}. For an integer $r\geq 1$, a $($multi$)$-hypergraph is {\em $r$-uniform} if every edge has size exactly $r$. The {\em degree} of a vertex $v$ of a multi-hypergraph $G$, denoted by $d_G(v)$, is the number of edges of $G$ containing $v$ and the {\em codegree} of two distinct vertices $u$ and $v$ of $G$ is the number of edges of $G$ containing both $u$ and $v$. A $($multi$)$-hypergraph $G = (A\cup B,E(G))$ is {\em bipartite with parts $A$ and $B$} if every edge of $G$ contains exactly one vertex from $A$. A {\em matching} of a  $($multi$)$-hypergraph $G$ is a set $M\subseteq E(G)$ such that $e_1\cap e_2=\emptyset$ for any $e_1,e_2\in M$. For a $($multi$)$-hypergraph $G$ and $A\subseteq V(G)$, we say a matching of $G$ is {\em $A$-perfect} if every vertex of $A$ is in an edge of the matching.

Delcourt and Postle \cite{DP2022} proved the following theorem.

\begin{Theorem}\label{thm:A_perfect_matching}{\rm \cite[Theorem 2.6]{DP2022}}
For any integer $r\geq 2$ and real $\beta> 0$, there exist an integer $D_{r,\beta}\geq 0$ and real $\alpha> 0$ such that the following holds for all $D\geq D_{r,\beta}$.
If $G =(A\cup B,E(G))$ is a bipartite $r$-uniform (multi)-hypergraph satisfying
\begin{enumerate}
 \item[$(1)$] every vertex in $A$ has degree at least $(1+D^{-\alpha})D$ and every vertex in $B$ has degree at most $D$, and
 \item[$(2)$] every pair of vertices of $G$ has codegree at most $D^{1-\beta}$,
\end{enumerate}
then there exists an $A$-perfect matching of $G$.
\end{Theorem}

\subsection{The asymptotic existence of Banff $(v,k,\lambda)$-BIBDs}\label{sec:Thm1}

We recall that a $(v,k,\lambda)$-BIBD is a Banff design if and only if it has a nesting.

\begin{proof}[{\bf Proof of Theorem \ref{thm:nested_BIBD}}]
Let $(X,\B)$ be a $(v,k,\lambda)$-BIBD and $\binom{X}{2}$ be the family of all $2$-subsets of $X$. We construct an auxiliary bipartite $(k+1)$-uniform (multi)-hypergraph $G=(V(G),E(G))$, where the vertex set $V(G)=\B\cup\binom{X}{2}$ are partitioned into two parts $\B$ and $\binom{X}{2}$, and the (multi)-set of edges is $E(G)=\{\{B,\{a,b_1\},\{a,b_2\},\dots,\{a,b_k\}\} \mid B=\{b_1,b_2,\dots,b_k\}\in \B, a\in X\setminus B\}$.

If there exists a $\B$-perfect matching $M$ of $G$, then the $(v,k,\lambda)$-BIBD $(X,\B)$ has a nesting. Indeed, for any $B=\{b_1,\dots,b_k\}\in\B$, there is a hyperedge $e_{B}=\{B,\{a_B,b_1\},\dots,\{a_B,b_k\}\}\in M$ for some $a_B\in X\setminus B$. Let $\phi(B)=a_B$. Since $M$ is a $\B$-perfect matching of $G$, $e_B\setminus\{B\}$, $B\in\B$, are mutually disjoint. Hence, $(X,\{B\cup\{\phi(B)\}\mid B\in\mathcal B\})$ is a $(v,k+1,\lambda+1)$-packing.

For any $B\in \B$, $d_{G}(B)=v-k$. Since $(X,\B)$ is a $(v,k,\lambda)$-BIBD, the number of blocks in $\B$ containing any given point is $\frac{\lambda(v-1)}{k-1}$. For any pair $\{a,b\}\in \binom{X}{2}$, the number of blocks in $\B$ containing $a$ but not containing $b$ is $\frac{\lambda(v-1)}{k-1}-\lambda$. Thus $d_G(\{a,b\})=2(\frac{\lambda(v-1)}{k-1}-\lambda)=2\frac{\lambda(v-k)}{k-1}$.
For any $B\in\B$ and $\{a,b\}\in\binom{X}{2}$, the codegree of $B$ and $\{a,b\}$ in $G$ is at most $1$.
For all distinct $\{a_1,b_1\}$ and $\{a_2,b_2\}$, if $\{a_1,b_1\}\cap\{a_2,b_2\}=\emptyset$, then the codegree of $\{a_1,b_1\}$ and $\{a_2,b_2\}$ in $G$ is $0$; if $\{a_1,b_1\}\cap\{a_2,b_2\}=\{a_1\}=\{a_2\}$, then by the definition of a $(v,k,\lambda)$-BIBD, there exist $\lambda$ blocks containing $\{b_1,b_2\}$, and so the codegree of $\{a_1,b_1\}$ and $\{a_2,b_2\}$ in $G$ is at most $\lambda$.

For a sufficiently large $v$, let $D=\frac{4\lambda(v-k)}{4\lambda+1}$. Given any real $\alpha>0$, we have $$d_{G}(B)=v-k=(1+\frac{1}{4\lambda})D\geq(1+D^{-\alpha})D$$
for $B\in \B$, and since $k\geq 2\lambda+2$,
$$d_G(\{a,b\})=2\frac{\lambda(v-k)}{k-1}\leq \frac{4\lambda(v-k)}{4\lambda+1}= D$$ for $\{a,b\}\in \binom{X}{2}$.
Given any real $\beta>0$, the codegree of $G$ is at most $\lambda\leq D^{1-\beta}$. Then applying Theorem \ref{thm:A_perfect_matching}, we have a $\B$-perfect matching $M$ of $G$, so the $(v,k,\lambda)$-BIBD $(X,\B)$ has a nesting.
\end{proof}

\subsection{The asymptotic existence of $(k_2,\lambda_2;k_1, \lambda_1)$-nestings}\label{sec:Thm2}

To establish the asymptotic existence of $(k_2,\lambda_2;k_1,\lambda_1)$-nestings, we need an asymptotic existence theorem from \cite{lw} on decompositions of edge-$r$-colored complete digraphs into prescribed edge-$r$-colored subgraphs. Here, {\em edge-$r$-colored} means that each edge has a color chosen from a set of $r$ colors.

Let $K_{n}^{[\lambda_{1},\lambda_{2},\ldots,\lambda_{r}]}$ denote the complete (multi)-digraph on $n$ vertices where there are exactly $\lambda_{i}$ edges of color $i$ joining $x$ to $y$ for every ordered pair $(x,y)$ of distinct vertices and every $1\leq i\leq r$. The digraph $K_{n}^{[\lambda_{1},\lambda_{2},\ldots,\lambda_{r}]}$ has $n(n-1)\sum_{i=1}^r \lambda_i$ edges.

A family $\cal S$ of subgraphs of a graph $\Omega$ is called a {\em decomposition} of $\Omega$ if every edge of $\Omega$ belongs to exactly one member of $\cal S$. 
Given a family $\cal G$ of edge-$r$-colored digraphs, a {\em $\cal G$-decomposition} of a graph $\Omega$ is a decomposition $\cal S$ of $\Omega$ such that every graph $S\in \cal S$ is isomorphic to some graph $G\in \cal G$. Here, we require that isomorphisms between edge-$r$-colored digraphs preserve the colors of edges.

For a vertex $x$ of an edge-$r$-colored digraph $G$, the {\em degree-vector} of $x$ is the $2r$-vector
\begin{equation*}
 \mathbf{d}_G(x) = (in_{1}(x), out_{1}(x), in_{2}(x), out_{2}(x),\ldots, in_{r}(x), out_{r}(x)),
\end{equation*}
where $in_{i}(x)$ and $out_{i}(x)$ denote, respectively, the in-degree and out-degree of vertex $x$ in the subgraph of $G$ spanned by the edges of color $i$, $1 \leq i \leq r$.

Let $\mathbf{r}=(\lambda_{1},\lambda_{2},\ldots,\lambda_{r})$ be a vector of $r$ positive integers and $\bar{\mathbf{r}}= (\lambda_{1},\lambda_{1},\lambda_{2},\lambda_{2},\ldots,$ $\lambda_{r},\lambda_{r})$ be a $2r$-vector. Let $\cal G$ be a family of edge-$r$-colored digraphs. Define $\alpha(\cal G, \mathbf{r})$ to be the greatest common divisor of the integers $t$ such that $t\bar{\mathbf{r}}$ is an integral linear combination of $\mathbf{d}_G(x)$ as $x$ runs over all vertices of all graphs in $\cal G$. Equivalently, $\alpha(\cal G, \mathbf{r})$ is the least positive integer $t_0$ such that $t_0\bar{\mathbf{r}}$ is an integral linear combination of $\mathbf{d}_G(x)$ as $x$ runs over all vertices of all graphs in $\cal G$.

For each $G\in {\cal G}$, let $\mu(G)=(w_{1},w_{2},\ldots,w_{r})$, where $w_{i}$ is the number of edges of color $i$ in $G$. Let $\beta({\cal G},\mathbf{r})$ be the greatest common divisor of the integers $m$ such that $m\mathbf{r}$ is an integral linear combination of $\mu(G)$, $G \in \cal G$. Equivalently, $\beta({\cal G},\mathbf{r})$ is the least positive integer $m$ such that $m\mathbf{r}$ is an integral linear combination of $\mu(G)$, $G \in \cal G$. ${\cal G}$ is called {\em $\mathbf{r}$-admissible} when the vector $\mathbf{r}$ is a positive rational linear combination of $\mu(G)$, $G \in \cal G$.

An edge-$r$-colored digraph is called {\em simple}, if it has no loops and there is at most one edge directed from $x$ to $y$ for each order pair $(x, y)$ of distinct vertices.

Now we state the powerful theorem due to Lamken and Wilson \cite {lw}. It has been employed to construct many combinatorial configurations (cf. \cite{Lamken09,Lamken15,WCF19}).

\begin{Theorem}{\rm \cite[Theorem 13.1]{lw}}\label{thm:decompositions_edge_colored_diKn}
Let $\cal G$ be an $\mathbf{r}$-admissible family of simple edge-$r$-colored digraphs,
where $\mathbf{r} = (\lambda_{1},\lambda_{2},\ldots,\lambda_{r})$. Then there exists a constant $n_{0} = n_{0}({\cal G},\mathbf{r})$ such that $\cal G$-decompositions of
$K_{n}^{[\lambda_{1},\lambda_{2},\ldots,\lambda_{r}]}$ exist for all $n\geq n_{0}$ satisfying $n - 1 \equiv 0 \pmod {\alpha({\cal G},\mathbf{r})}$ and $n(n-1) \equiv 0 \pmod {\beta({\cal G},\mathbf{r})}$.
\end{Theorem}

\begin{proof}[{\bf Proof of Theorem \ref{thm:(k_1__k_2)_nested}}]
Suppose that $U_2$ is a set of $k_2$ vertices and $U_1\subseteq U_2$ is a $k_1$-subset. Let $G$ be an edge-$2$-colored complete digraph with vertex set $U_2$, where the edges within $U_1$ are colored with color $1$, and the edges within $U_2\setminus U_1$ and between $U_2\setminus U_1$ and $U_1$ are colored with color $2$.

Let $\mathcal G=\{G\}$. We claim that if there is a $\cal G$-decomposition $\mathcal S$ of $K_{v}^{[\lambda_1,\lambda_2-\lambda_1]}$, then there exists a $(k_2,\lambda_2;k_1, \lambda_1)$-nesting.
Indeed, $(X,\{V(S)\mid S\in\mathcal S\})$ is a $(v,k_2,\lambda_2)$-BIBD, where $X$ is the vertex set of $K_{v}^{[\lambda_1,\lambda_2-\lambda_1]}$ and $V(S)$ is the vertex set of $S$ in $\mathcal S$. For each $S\in\mathcal S$, let $S'$ be the subgraph of $S$ obtained by keeping only the directed edges with colour $1$. Note that $\{S'\mid S\in\mathcal S\}$ is a $\{K_{k_1}^{[1]}\}$-decomposition of $K_{v}^{[\lambda_1]}$, and so $(X,\{V(S')\mid S\in\mathcal S\})$ is a $(v,k_1,\lambda_1)$-BIBD. Thus we obtain a $(k_2,\lambda_2;k_1, \lambda_1)$-nesting.


By the definition of $G$, $\mathbf{d}_G(u_1)=(k_1-1,k_1-1,k_2-k_1,k_2-k_1)$ for each $u_1\in U_1$ and $\mathbf{d}_G(u_2)=(0,0,k_2-1,k_2-1)$ for each $u_2\in U_2\setminus U_1$. Note that $\mathbf{r}=(\lambda_1,\lambda_2-\lambda_1)$ and $\bar{\mathbf{r}}=(\lambda_1,\lambda_1,\lambda_2-\lambda_1,\lambda_2-\lambda_1)$. If $t(\lambda_1,\lambda_1,\lambda_2-\lambda_1,\lambda_2-\lambda_1)=\alpha_1(k_1-1,k_1-1,k_2-k_1,k_2-k_1)
+\alpha_2(0,0,k_2-1,k_2-1)$, where $t$ is a positive integer, and $\alpha_1$ and $\alpha_2$ are integers, then
$$t\lambda_1=\alpha_1(k_1-1) \text{~~and~~} t(\lambda_2-\lambda_1)=\alpha_1(k_2-k_1)+\alpha_2(k_2-1),$$
which implies
$$t\lambda_1=\alpha_1(k_1-1) \text{~~and~~} t\lambda_2=(\alpha_1+\alpha_2)(k_2-1).$$
Thus
$$\alpha_1=\frac{t\lambda_1}{k_1-1} \text{~~and~~} \alpha_2=\frac{t\lambda_2}{k_2-1}-\frac{t\lambda_1}{k_1-1}.$$
Let
$$I_1=\left\{s\in\mathbb{Z}^+\mid \frac{s\lambda_1}{k_1-1}\in\mathbb{Z}\text{~and~}\frac{s\lambda_2}{k_2-1}\in\mathbb{Z}\right\}.$$
Then $t\in I_1$ and $\alpha(\mathcal G, \mathbf{r})$ is the smallest positive integer in $I_1$.
Since $\lambda_1 (v-1)\equiv 0\pmod{k_1-1}$ and $\lambda_2 (v-1)\equiv 0\pmod{k_2-1}$, $\frac{\lambda_1(v-1)}{k_1-1}$ and $\frac{\lambda_2(v-1)}{k_2-1}$ are both integers and so $v-1\in I_1$. If $v-1\not\equiv 0\pmod{\alpha(\mathcal G, \mathbf{r})}$, then there exist integers $q$ and $r$ such that $v-1=q\cdot\alpha(\mathcal G, \mathbf{r})+r$ where $0<r<\alpha(\mathcal G, \mathbf{r})$. Since $v-1\in I_1$ and $\alpha(\mathcal G, \mathbf{r})\in I_1$, we have $r\in I_1$. This yields a contradiction since $\alpha(\mathcal G, \mathbf{r})$ is the smallest positive integer in $I_1$. So $v-1\equiv 0\pmod{\alpha(\mathcal G, \mathbf{r})}$.

By the definition of $G$, $\mu(G)=(k_1(k_1-1),k_2(k_2-1)-k_1(k_1-1))$. If $m(\lambda_1,\lambda_2-\lambda_1)=\beta(k_1(k_1-1),k_2(k_2-1)-k_1(k_1-1))$, where $m$ is a positive integer and $\beta$ is an integer, then
$$m\lambda_1=\beta k_1(k_1-1)\text{~~and~~}m\lambda_2=\beta k_2(k_2-1).$$
By the assumption $\lambda_1k_2(k_2-1)=\lambda_2k_1(k_1-1)$, the first equation $m\lambda_1=\beta k_1(k_1-1)$ above implies the second one $m\lambda_2=\beta k_2(k_2-1)$. Let
$$I_2=\left\{\frac{\gamma k_1(k_1-1)}{\lambda_1}\in\mathbb{Z}^+\mid \gamma\in\mathbb{Z}\right\}.$$
Then $m\in I_2$ and $\beta(\mathcal G, \mathbf{r})$ is the smallest positive integer in $I_2$.
By $\lambda_1 v(v-1)\equiv 0\pmod{k_1(k_1-1)}$, there is an integer $\gamma$ such that $\lambda_1 v(v-1)=\gamma k_1(k_1-1)$, and hence $v(v-1)\in I_2$. If $v(v-1)\not\equiv 0\pmod{\beta(\mathcal G, \mathbf{r})}$, then there exist integers $q$ and $r$ such that $v(v-1)=q\cdot\beta(\mathcal G, \mathbf{r})+r$ where $0<r<\beta(\mathcal G, \mathbf{r})$. Since $v(v-1)\in I_2$ and $\beta(\mathcal G, \mathbf{r})\in I_2$, we have $r\in I_2$. A contradiction occurs by the definition of $\beta(\mathcal G, \mathbf{r})$. Thus $v(v-1)\equiv 0\pmod{\beta(\mathcal G, \mathbf{r})}.$

By the assumption $\lambda_1k_2(k_2-1)=\lambda_2k_1(k_1-1)$, we have $$(\lambda_1,\lambda_2-\lambda_1)=\frac{\lambda_1}{k_1(k_1-1)}(k_1(k_1-1),k_2(k_2-1)-k_1(k_1-1)),$$
so $\cal G$ is an $\mathbf{r}$-admissible family consisting of only one simple edge-$2$-colored digraph.

Now we can apply Theorem \ref{thm:decompositions_edge_colored_diKn} to obtain a $\cal G$-decomposition of $K_{v}^{[\lambda_1,\lambda_2-\lambda_1]}$, which gives a $(k_2,\lambda_2;k_1,\lambda_1)$-nesting.
\end{proof}

\section{Banff difference families}\label{sec:banff DF}

The following lemma will be used to prove Theorem \ref{thm:DF-Banff_DF}. Let $G$ be a finite group. We recall that for any $a\in G$ and $B=\{b_1,\dots,b_k\}\subseteq G$, let $-B=\{-b_1,\dots,-b_k\}$ and $B+a=\{b_1+a,\dots,b_k+a\}$.

\begin{Lemma}\label{lem:solution_number}
Let $C$ be a constant and $G$ be a finite abelian group such that the number of $2$-order elements in $G$ is at most $C$. Let $k\geq 1$. If $B$ is a $k$-subset of $G$, then
$$|\{a\in G\mid -(B+a)\cap (B+a)\neq\emptyset\}|\leq2^Ck^2.$$
\end{Lemma}

\begin{proof}
If $-(B+a)\cap (B+a)\neq\emptyset$, then $-(b_1+a)=b_2+a$ for some $b_1,b_2\in B$, and so $b_1+b_2+2a=0$.
Since the number of $2$-order elements in $G$ is at most $C$, $|\{a\in G\mid b_1+b_2+2a=0\}|\leq2^C$ for any given $b_1,b_2\in G$. Therefore, $|\{a\in G\mid -(B+a)\cap (B+a)\neq\emptyset\}|\leq2^Ck^2$.
\end{proof}

\begin{proof}[{\bf Proof of Theorem \ref{thm:DF-Banff_DF}}]
Let $G$ be an abelian group of order $v$, $\B$ be a $(G,k,\lambda)$-DF and $G'=\{\{x,-x\}\mid x\in G\setminus\{0\},-x\neq x\}$. We construct an auxiliary bipartite $(k+1)$-uniform (multi)-hypergraph $H=(V(H),E(H))$, where the vertex set $V(H)=\mathcal B\cup G'$ are partitioned into two parts $\mathcal B$ and $G'$, and the (multi)-set of edges is $E(H)=\{\{B\}\cup \{\{b_1+a,-b_1-a\},\dots,\{b_k+a,-b_k-a\}\} \mid a\in G, B=\{b_1,\dots,b_k\}\in\mathcal B\ \text{and} -(B+a)\cap(B+a)=\emptyset\}$.

For any $e\in E(H)$, there exist $B_e=\{b_1,\dots,b_k\}\in \mathcal B$ and $a_e\in G$ such that $e=\{B_e\}\cup \{\{b_1+a_e,-b_1-a_e\},\dots,\{b_k+a_e,-b_k-a_e\}\}$ and let $S(e)=B_e+a_e$.

If there exists a $\mathcal B$-perfect matching $M$ of $H$, then $\mathcal F=\{S(e)\mid e\in M\}$ is a $(G,k,\lambda)$-BDF. Indeed, for any $B=\{b_1,\dots,b_k\}\in\mathcal B$, there is a hyperedge $e_B=\{B,\{b_1+a_B,-b_1-a_B\},\dots,\{b_k+a_B,-b_k-a_B\}\}\in M$ for some $a_B\in G$. By the definition of hyperedges in $H$, for all $B\in\mathcal B$, $-(B+a_B)\cap(B+a_B)=\emptyset$. Since $M$ is a $\mathcal B$-perfect matching of $H$, $e_B\setminus\{B\}$, $B\in\mathcal B$, are mutually disjoint. That is, for any $e,e'\in M$, $S(e)$, $-S(e)$, $S(e')$ and $-S(e')$ are mutually disjoint. Because $M$ is a $\mathcal B$-perfect matching of $H$, we can find unique one hyperedge $e\in M$ with $B\in e$ for any $B\in\mathcal B$ and so $|\mathcal F|=|\mathcal B|$. Since $\mathcal B$ is a $(G,k,\lambda)$-DF, $\mathcal F$ is a $(G,k,\lambda)$-BDF.

For any $B\in \B$, $d_{H}(B)$ is equal to the number of $a\in G$ such that $-(B+a)\cap(B+a)=\emptyset$. By Lemma \ref{lem:solution_number}, $d_{H}(B)\geq v-2^Ck^2$. For every $\{x,-x\}\in G'$, $d_H(\{x,-x\})$ is at most the number of triples $(B,b,a)$'s satisfying $b\in B\in\mathcal B$, $a\in G$, and $b+a=x$ or $-b-a=x$. Since $\mathcal B$ is a $(G,k,\lambda)$-DF, $|\mathcal B|=\frac{\lambda(v-1)}{k(k-1)}$ and so $d_H(\{x,-x\})\leq\frac{\lambda(v-1)}{k(k-1)}\times 2k=\frac{2\lambda(v-1)}{k-1}$. For every $B\in\mathcal B$ and $\{x,-x\}\in G'$, the codegree of $B$ and $\{x,-x\}$ in $H$ is at most $2k$.
For all distinct $\{x_1,-x_1\}$ and $\{x_2,-x_2\}$ in $G'$, since $\mathcal B$ is a $(G,k,\lambda)$-DF, it yields a $(v,k,\lambda)$-BIBD, and so there exist exactly $\lambda$ pairs $(B,a)$'s with $B\in \B$ and $a\in G$ such that $\{y_1,y_2\}\in B+a$ for every $\{y_1,y_2\}\in\{\{x_1,x_2\},\{x_1,-x_1\},\{x_1,-x_2\},\{-x_1,-x_2\}\}$. Thus the codegree of $\{x_1,-x_1\}$ and $\{x_2,-x_2\}$ in $H$ is at most $4\lambda$.

For a sufficiently large $v$, let $D=\frac{4\lambda}{2k-3}(v-2^Ck^2)$. Since $k\geq2\lambda+2$, $\frac{2k-3}{4\lambda}\geq1+\frac{1}{4\lambda}$.
Given any real $\alpha>0$, for every $B\in \mathcal B$,
$$d_{H}(B)\geq v-2^Ck^2=\frac{2k-3}{4\lambda}D\geq(1+\frac{1}{4\lambda})D\geq(1+D^{-\alpha})D.$$
Since $v$ is sufficiently large, $\frac{v-k}{k-1}\leq\frac{2(v-2^Ck^2)}{2k-3}$ and so $$d_H(\{x,-x\})\leq \frac{2\lambda(v-k)}{k-1}\leq \frac{4\lambda}{2k-3}(v-2^Ck^2)=D$$ for any $\{x,-x\}\in G'$.
Given any real $\beta>0$, the codegree of $H$ is at most $\max\{2k,4\lambda\}\leq D^{1-\beta}$. Then applying Theorem \ref{thm:A_perfect_matching}, there is a $\mathcal B$-perfect matching $M$ of $H$, and so we obtain a $(G,k,\lambda)$-BDF $\mathcal F$.
\end{proof}

\section{A solution to Nov\'{a}k's conjecture on cyclic BIBDs}\label{sec:Nov}

For any block $B$ of a cyclic $(v, k, \lambda)$-BIBD, the {\em block orbit} containing $B$ is the set of distinct blocks $B+i$, $i\in\Z_v$. If a block orbit has $v$ blocks, the block orbit is {\em full}. Otherwise, it is {\em short}.


\begin{Lemma}\label{lem:number_short_orbits}{\rm \cite[Lemma 3]{FHW2021}}
Let $k\geq2$ and $\lambda\geq1$ be fixed integers.
If $(V,\mathcal B)$ is a cyclic $(v, k, \lambda)$-BIBD with $h$ short orbits and $m$ full orbits, then
\begin{enumerate}
  \item[$(1)$] $h\leq2\lambda\sqrt{k}$; and
  \item[$(2)$] $\frac{\lambda(v-1)}{k(k-1)}-2\lambda\sqrt{k}\leq m\leq\frac{\lambda(v-1)}{k(k-1)}$.
\end{enumerate}
\end{Lemma}

\begin{Lemma}\label{lem:forbid_number}
Let $k\geq 1$. If $B$ is a $k$-subset of $\mathbb{Z}_v$ and $X\subseteq \mathbb{Z}_v$, then
$$|\{a\in \mathbb{Z}_v\mid X\cap (B+a)\neq\emptyset\}|\leq k|X|.$$
Moreover, if $v>k|X|$, then there exists $a\in \mathbb{Z}_v$ such that $X\cap (B+a)=\emptyset$.
\end{Lemma}

\begin{proof}
The case $X=\emptyset$ is trivial. Assume that $X\not=\emptyset$. If $X\cap (B+a)\neq\emptyset$, then there exist $x\in X$ and $b\in B$ such that $x=b+a$. Since $|\{a\in \mathbb{Z}_v\mid x=b+a\}|=1$ for any given $b,x\in \mathbb{Z}_v$, $|\{a\in \mathbb{Z}_v\mid X\cap (B+a)\neq\emptyset\}|\leq k|X|$.
\end{proof}

\begin{proof}[{\bf Proof of Theorem \ref{thm:novak_conj_solution-main}}]
Let $(\mathbb{Z}_v,\mathcal A)$ be a cyclic $(v,k,\lambda)$-BIBD. We pick any one of blocks from each orbit of this BIBD to form a collection $\mathcal B_1\cup \mathcal B_2$ of base blocks, where $\mathcal B_1$ and $\mathcal B_2$ are the (multi)-set of base blocks coming from short and full orbits, respectively. By Lemma \ref{lem:number_short_orbits}, $|\mathcal B_1|\leq2\lambda\sqrt{k}$.

If $\mathcal B_1\neq\emptyset$, assuming that ${\cal B}_1=\{B_1,\ldots,B_h\}$, then there exist $a_1,\ldots,a_h\in \mathbb{Z}_v$ such that $B_i+a_i$, $1\leq i\leq h$, are mutually disjoint. Indeed, we can find these elements $a_i$, $1\leq i\leq h$, greedily. We can take $a_1=0$ and write $T_1=B_1$. In general, given $a_1,\dots,a_{i-1}$ and $T_{i-1}$ for $i\geq 2$, applying Lemma \ref{lem:forbid_number} with $X=T_{i-1}$ and $B=B_i$, there is $a_i\in \mathbb{Z}_v$ such that $T_{i-1}\cap (B_i+a_i)=\emptyset$ as $v$ is sufficiently large. Take $T_i=T_{i-1}\cup (B_i+a_i)$. Since $v$ is sufficiently large and $h=|\mathcal B_1|$ is a constant depending only on $\lambda$ and $k$, there exist $a_1,\dots,a_h\in \mathbb{Z}_v$ such that $B_i+a_i$, $1\leq i\leq h$, are mutually disjoint. Let $\mathcal F_1=\{B_i+a_i:1\leq i\leq h\}$ and $T=T_h$. Note that $|T|=kh$ is a constant. If $\mathcal B_1=\emptyset$, then take $\mathcal F_1=\emptyset$ and $T=\emptyset$.

Let $G'=\mathbb{Z}_v\setminus T$. We construct an auxiliary bipartite $(k+1)$-uniform (multi)-hypergraph $H=(V(H),E(H))$, where the vertex set $V(H)=\B_2\cup G'$ are partitioned into two parts $\B_2$ and $G'$, and the (multi)-set of edges is $E(H)=\{\{B,b_1+a,,\dots,b_k+a\} \mid a\in \mathbb{Z}_v, B=\{b_1,\dots,b_k\}\in\mathcal B_2, T\cap(B+a)=\emptyset\}$.

For any $e\in E(H)$, there exist $B_e=\{b_1,\dots,b_k\}\in \mathcal B_2$ and $a_e\in \mathbb{Z}_v$ such that $e=\{B_e,b_1+a_e,\dots,b_k+a_e\}$ and let $S(e)=B_e+a_e$.

If there exists a $\mathcal B_2$-perfect matching $M$ of $H$, then $\mathcal F_1\cup\mathcal F_2$ is our desired structure, where $\mathcal F_2=\{S(e):e\in M\}$. Indeed, for any $B=\{b_1,\dots,b_k\}\in\mathcal B_2$, there is a hyperedge $e_B=\{B,b_1+a_B,,\dots,b_k+a_B\}\in M$ for some $a_B\in \mathbb{Z}_v$. By the definition of hyperedges in $H$, $T\cap(B+a_B)=\emptyset$ for all $B\in\mathcal B_2$. That is, $B'\cap(B+a_B)=\emptyset$ for all $B'\in\mathcal F_1$ and $B\in\mathcal B_2$. Since $M$ is a $\mathcal B_2$-perfect matching of $H$, $e_B\setminus\{B\}$, $B\in\mathcal B_2$, are mutually disjoint. That is, for any $B,B'\in\mathcal B_2$ and $e,e'\in M$ with $B\in e$ and $B'\in e'$, $S(e)$ and $S(e')$ are disjoint.

For any $B\in \B_2$, $d_{H}(B)$ is equal to the number of $a\in \mathbb{Z}_v$ satisfying $T\cap(B+a)=\emptyset$. By Lemma \ref{lem:forbid_number}, $d_{H}(B)\geq v-k|T|$.
For every $x\in G'$, $d_H(x)$ is at most the number $(B,b,a)$'s satisfying $b\in B\in\mathcal B_2$, $a\in \mathbb{Z}_v$ and $b+a=x$. It follows from Lemma \ref{lem:number_short_orbits}(2) that $|\mathcal B_2|\leq \frac{\lambda(v-1)}{k(k-1)}$ and so $d_H(x)\leq\frac{\lambda(v-1)}{k(k-1)}\times k=\frac{\lambda(v-1)}{k-1}$. For every $B\in\mathcal B_2$ and $x\in G'$, the codegree of $B$ and $x$ in $H$ is at most $k$.
Since $\{B+a\mid B\in\mathcal B_2,a\in \mathbb{Z}_v\}\subseteq\mathcal A$ and $(\mathbb{Z}_v,\mathcal A)$ is a cyclic $(v,k,\lambda)$-BIBD,  there exist at most $\lambda$ pairs $(B,a)$'s with $B\in \B_2$ and $a\in \mathbb{Z}_v$ such that $\{x_1,x_2\}\subseteq B+a$ for any distinct $x_1$ and $x_2$ in $G'$. Thus the codegree of $x_1$ and $x_2$ in $H$ is at most $\lambda$.

For a sufficiently large $v$, let $D=\frac{2\lambda}{2k-3}(v-k|T|)$. Since $k\geq \lambda+2$, $\frac{2k-3}{2\lambda}\geq 1+\frac{1}{2\lambda}$. Given any real $\alpha>0$, for every $B\in \mathcal B_2$, we have
$$d_{H}(B)\geq v-k|T|=\frac{2k-3}{2\lambda}D\geq(1+\frac{1}{2\lambda})D\geq(1+D^{-\alpha})D.$$
Since $v$ is sufficiently large, $\frac{v-1}{k-1}\leq\frac{2(v-k|T|)}{2k-3}$ and hence
$$d_H(x)\leq \frac{\lambda(v-1)}{k-1}\leq \frac{2\lambda}{2k-3}(v-k|T|)=D$$ for any $x\in G'$.
Given any real $\beta>0$, the codegree of $H$ is at most $\max\{k,\lambda\}\leq D^{1-\beta}$. Then apply Theorem \ref{thm:A_perfect_matching} to obtain a $\mathcal B_2$-perfect matching $M$ of $H$.
\end{proof}

\section{Concluding remarks}

This paper investigates the asymptotic existence of $(v,k,\lambda)$-BIBDs having a (perfect) nesting. A $(v,k,\lambda)$-BIBD has a nesting if and only if its Levi graph has a harmonious coloring with $v$ colors. Theorem \ref{thm:nested_BIBD} shows that every $(v,k,\lambda)$-BIBD can be nested into a $(v,k+1,\lambda+1)$-packing for any $k\geq 2\lambda+2$ and large enough $v$.
Theorem \ref{thm:perfect_nested} shows that for $k=2\lambda+1$, for any sufficiently large $v$ satisfying $v \equiv 1 \pmod {2k}$, there exists a $(v,k,\lambda)$-BIBD having a perfect nesting. A natural and interesting problem is whether every $(v,k,\lambda)$-BIBD for $k=2\lambda+1$ has a perfect nesting when $v$ is sufficiently large.


Theorem \ref{thm:DF-Banff_DF} shows that if $G$ is a finite abelian group with a large size whose number of $2$-order elements is no more than a given constant, and $k\geq 2\lambda+2$, then a $(G,k,\lambda)$-BDF can be obtained by taking any $(G,k,\lambda)$-DF and then
replacing each of its base blocks by a suitable translation. This gives an asymptotic solution to a conjecture in \cite[Section 3]{BMNR23}, where it is conjectured that every $(\mathbb{Z}_v,k,1)$-DF can yield a $(\mathbb{Z}_v,k,1)$-BDF.

The condition that the number of $2$-order elements in $G$ is at most a given constant $C$ in Theorem \ref{thm:DF-Banff_DF} is necessary. For example, every $((\mathbb{F}_{2^n},+),k,\lambda)$-DF cannot produce an $((\mathbb{F}_{2^n},+),k,\lambda)$-BDF since $B=-B$ for any $k$-subset $B\subseteq \mathbb{F}_{2^n}$.

Theorem \ref{thm:novak_conj_solution-main} gives a solution to Conjecture \ref{conj:novak_conj_gener} when $k\geq \lambda+2$. There is a much stronger conjecture to be considered. It was conjectured in \cite[Conjecture 2]{FHW2021} that for any cyclic $(v,k,1)$-BIBD (here $v$ is not required to be large enough but $\lambda=1$), it is always possible to choose one block from each block orbit so that the chosen blocks are pairwise disjoint. This conjecture was confirmed for any prime $v\equiv 1\pmod{k(k-1)}$ by using the Combinatorial Nullstellensatz in \cite[Theorem 1]{FHW2021}.

\section*{Acknowledgements}

We are grateful to Professor Marco Buratti for his valuable discussions. We thank Professor Douglas R. Stinson for his helpful comments which motivated us to give Section \ref{sec:Thm2}.

\section*{Conflict of Interest}

The authors declare that there are no conflicts of interest that are relevant to the content of this article.

\section*{Data availability}

Data sharing not applicable to this article as no datasets were generated or analysed during the current study.

\end{document}